\documentclass{rmmcart}
\usepackage[dvips]{graphicx}
\usepackage{amssymb,amsmath,amsthm}

\newtheorem{theorem}{Theorem}[section]

\newtheorem{proposition}[theorem]{Proposition}
\theoremstyle{definition}

\newtheorem{conjecture}[theorem]{Conjecture}

\newcommand*{\conjlab}[1]{\label{conj:#1}}
\newcommand*{\conjref}[1]{Conjecture \ref{conj:#1}}

\newcommand*{\proplab}[1]{\label{prop:#1}}
\newcommand*{\propref}[1]{Proposition \ref{prop:#1}}

\newcommand*{\thmlab}[1]{\label{thm:#1}}

\newcommand*{\tablelab}[1]{\label{table:#1}}

\def\b0{{\bf 0}}

\def\calb{\mathcal{B}}

\def\calg{\mathcal{G}}

\def\calp{\mathcal{P}}

\def\calw{\mathcal{W}}

\def\set#1{\{#1\}}

     \title{Gr\"obner bases of toric ideals associated with matroids}

     \author{Ken-ichi Hayase}
     \address{Department of Pure and Applied Mathematics,
Graduate school of Information Science and Technology,
Osaka University,
Suita, Osaka 565-0871, Japan.}
     \email{k-hayase@ist.osaka-u.ac.jp}
     \thanks{}

     \author{Takayuki Hibi}
     \address{Department of Pure and Applied Mathematics,
Graduate school of Information Science and Technology,
Osaka University,
Suita, Osaka 565-0871, Japan.}
     \email{hibi@math.sci.osaka-u.ac.jp}
     \thanks{}
    \author{Koyo Katsuno}
     \address{Department of Pure and Applied Mathematics,
Graduate school of Information Science and Technology,
Osaka University,
Suita, Osaka 565-0871, Japan.}
     \email{k-katsuno@ist.osaka-u.ac.jp}
     \thanks{}
    \author{Kazuki Shibata}
     \address{Department of Mathematics,
Collage of Science,
Rikkyo University,
Toshima-ku, Tokyo 171-8501, Japan.}
     \email{k-shibata@rikkyo.ac.jp}
     \thanks{}

     \date{}

     \keywords{Toric ideals, Gr\"obner bases and matroids}
     \subjclass{Primary 13P10, Secondary 05B35}

\begin{document}
\begin{abstract}
In 1980, White conjectured that the toric ideal of a matroid is generated by quadratic binomials corresponding to a symmetric exchange.
In this paper, we compute Gr\"obner bases of toric ideals associated with matroids and show that, for every matroid on ground sets of size at most seven except for two matroids, Gr\"obner bases of toric ideals consist of quadratic binomials corresponding to a symmetric exchange.
\end{abstract}
     \maketitle

\section*{Introduction}
A {\it matroid} $M$ is a pair $(E,\calb)$, where $
E$ is a finite set and $\calb$ is a non-empty collection of subsets of $E$, which satisfies the following axiom known as the {\it symmetric exchange property}.
\begin{itemize}
\item For any $B$ and $B^{'}$ in $\calb$, for any $\alpha \in B \setminus B^{'}$, there exists $\beta \in B^{'} \setminus B$ such that both $(B \cup \{ \beta \}) \setminus \{ \alpha \}$ and $(B^{'} \cup \{ \alpha \}) \setminus \{ \beta \}$ are also members of $\calb$.
\end{itemize}
If $M$ is the matroid $(E,\calb)$, then $M$ is called a matroid {\it on} $E$.
The member of $\calb$ is called a {\it basis} of $M$ and $E$ is called a {\it ground set} of $M$.
All the members of $\calb$ have the same cardinality.
This cardinality is said to be the {\it rank} of $M$ and is denoted as $r(M)$.
We write $\calb(M)$ for $\calb$ and $E(M)$ for $E$.
For a matroid $M$, let $S_M=K[x_{B}~|~B \in \calb(M)]$ and $K[T]=K[t_{i}~|~i \in E(M)]$ be two polynomial rings over a field $K$.
We consider the ring homomorphism $\pi_{M}~:~S_M \rightarrow K[T]$ defined by $x_{B} \mapsto \prod_{l \in B} t_l$.
The {\it toric ideal} $I_M$ of $M$ is the kernel of $\pi_{M}$ and the image of $\pi_M$ is called the {\it basis monomial ring} of $M$, and it was introduced by White \cite{White1}.
It is known that $I_M$ is generated by homogeneous binomials and reduced Gr\"obner bases of $I_M$ consist of homogeneous binomials (see \cite{Sturmfels1}).
White showed that the basis monomial ring is normal for any matroid \cite{White1} and presented the following conjecture.
\begin{conjecture}[\cite{White2}]
\conjlab{White_gene}
For any matroid $M$, $I_M$ is generated by the binomials $x_{B_i} x_{B_j} -x_{B_k} x_{B_l}$ such that the pair of bases $B_k$, $B_l$ can be obtained from the pair of bases $B_i$, $B_j$ by a symmetric exchange.
\end{conjecture}
\conjref{White_gene} is true for graphic matroids \cite{Blasiak}, matroids with $r(M) \le 3$ \cite{Kashiwabara}, sparse paving matroids \cite{Bonin} and strongly base orderable matroids \cite{LaMi}.
In the present paper, the following stronger conjecture will be studied:
\begin{conjecture}
\conjlab{White_GB}
For any matroid $M$, $I_M$ has a Gr\"obner basis consisting of the binomials $x_{B_i} x_{B_j} -x_{B_k} x_{B_l}$ such that the pair of bases $B_k$, $B_l$ can be obtained from the pair of bases $B_i$, $B_j$ by a symmetric exchange.
\end{conjecture}
\conjref{White_GB} is true for uniform matroids $U_{r,d}$ \cite{Sturmfels1}, matroids with $r(M) \le 2$ \cite{Blum1,OH}, base-sortable matroids \cite{Blum1} and lattice path matroids \cite{Schweig}.


In this paper, we compute a Gr\"obner basis of $I_M$ and show that \conjref{White_GB} is true for every matroid $M$ with $|E(M)| \le 7$ except for two matroids.

\section{Operations on matroids}
In this section, we show that \conjref{White_GB} is true for every matroid $M$ with $|E(M)| \le 6$.

Let $M$ be a matroid.
Then the {\it dual} of $M$, denoted as $M^{\ast}$, is a matroid on $E(M)$ whose collection of bases is $\{ E(M) \setminus B~|~ B \in \calb(M)\}$.
\begin{proposition}[\cite{Blum1,White1}]
\proplab{closed}
The class of matroids such that \conjref{White_GB} is true is closed under the duality.
\end{proposition}
By \propref{closed}, it follows that \conjref{White_GB} is true for every matroid $M$ with $|E(M)| \le 5$ and, in order to prove that \conjref{White_GB} is true, we may assume that $3 \le r(M) \le \lfloor \frac{|E(M)|}{2} \rfloor$.
For two matroids $M_1$ and $M_2$ with $E(M_1) \cap E(M_2)=\emptyset$,
a matroid whose ground set is $E(M_1) \cup E(M_2)$ and whose collection of bases is
$
\{ B \cup D~|~ B \in \calb(M_1), D \in \calb(M_2)\}
$
is called the {\it direct sum} of $M_1$ and $M_2$ and is denoted as $M_1 \oplus M_2$.
Let $M_1$ and $M_2$ be matroids with $E(M_1) \cap E(M_2)=\{ p \}$ and assume that neither $M_1$ nor $M_2$ has $p$ as a loop or a coloop, where, for a matroid $M$, an element $p \in E(M)$ is called a {\it loop} (resp. a {\it coloop}) of $M$ if, for any $B \in \calb(M)$, $p \notin B$ (resp. $p \in B$).
Then the $2$-{\it sum} $M_1 \oplus_2 M_2$ of $M_1$ and $M_2$ is a matroid on $(E(M_1) \cup E(M_2)) \setminus \{p\}$ whose collection of bases is
$$
\{ (B \cup D) \setminus \{ p \}~|~ B \in \calb(M_1), D \in \calb(M_2), B\cap D=\emptyset,p \in B\cup D\}.
$$


A matroid $M$ is said to be $3$-{\it connected} if it can be written as neither a direct sum of two non-empty matroids nor a $2$-sum of two matroids on ground sets of size at least three \cite{Oxley}. 
\begin{proposition}[\cite{Blum1,Shibata}]
\proplab{closed2}
Let $M_1$ and $M_2$ be matroids.
If toric ideals $I_{M_1}$ and $I_{M_2}$ of two matroids $M_1$ and $M_2$ have a quadratic Gr\"obner basis,
then so do $I_{M_1 \oplus M_2}$ and $I_{M_1 \oplus_2 M_2}$.
\end{proposition}
By \propref{closed} and 1.2, we have
\begin{proposition}
\conjref{White_GB} is true for every matroid $M$ with $|E(M)|=6$.
\end{proposition}

\begin{proof}
Let $M$ be a matroid $|E(M)|=6$.
By \propref{closed}, we may assume that $r(M)=3$.
Since $r(M)=3$, any quadratic binomial in $I_M$ corresponds to a symmetric exchange \cite{Kashiwabara}.
Moreover if $M$ is not $3$-connected, then $M$ can be written as either $M=M_1 \oplus M_2$ or $M=M_1 \oplus_2 M_2$ for some two matroids $M_1$ and $M_2$.
Hence, by \propref{closed2}, it is enough to prove that the toric ideal $I_M$ of every $3$-connected matroid $M$ has a quadratic Gr\"obner basis.
By \cite[Corollary 12.2.19]{Oxley}, $M$ is one of $P_6$, $Q_6$, $\calw^3$, $M(K_4)$ and $U_{3,6}$.
Toric ideals $I_{\calw^3}$ and $I_{U_{3,6}}$ have a quadratic Gr\"obner basis since $\calw^3$ and $U_{3,6}$ are base-sortable \cite{Blum1,Sturmfels1}, and the toric ideal $I_{M(K_4)}$ has a quadratic Gr\"obner basis with respect to the lexicographic order on $S_{M(K_4)}$ \cite{Blum2}.
Moreover, by using Blum's results \cite{Blum1}, it follows that $P_6$ and $Q_6$ are base-sotable, in particular, $I_{P_6}$ and $I_{Q_6}$ have a quadratic Gr\"obner basis.
\end{proof}

\section{Gr\"obner bases of Toric ideals of matroids on ground sets of size seven}
In this section, for every $3$-connected matroid $M$ of $r(M)=3$ with $|E(M)|=7$, we compute a Gr\"obner basis of $I_M$.
For a matroid $M$, let
$$
J_M = \left\langle
x_{B}-\prod_{l \in B} t_l~|~B \in \calb(M)
\right\rangle
$$
be the ideal in $\widetilde{S}_M=K[x_{B},t_{i}~|~B \in \calb(M), i \in E(M)]$.
Then the toric ideal $I_M$ is equal to the intersection of $J_M$ and $S_M$ \cite[Lemma 1.5.11]{Hibi}.
Therefore if $\calg_M$ is a Gr\"obner basis for $J_M$ with respect to the lexicographic order on $\widetilde{S}_M$ induced by the ordering $t_i > x_{B}$ for any $i \in E(M)$ and $B \in \calb(M)$, then $\calg_M \cap S_M$ is a Gr\"obner basis for $I_M$.
On the other hand, there are $108$ matroids of rank $3$ on ground sets of size seven \cite{MayRoy} and all collections of bases are given at

\texttt{http://www-imai.is.s.u-tokyo.ac.jp/\textasciitilde ymatsu/matroid/index.html}.

\noindent
By using software package \texttt{MACAULAY2} \cite{Macaulay2}, we can check that the number of all $3$-connected matroids is $18$ (Table $1$).
Hence, in order to prove that \conjref{White_GB} is true for every matroid $M$ with $|E(M)|=7$, it is enough to prove that toric ideals of these $18$ $3$-connected matroids have a quadratic Gr\"obner basis.
\begin{table}[htbp]
\tablelab{lists}
\begin{center}
\begin{tabular}{c|l}
\hline
Matroid & A collection of bases of $M_i$ \\
\hline
\hline
$M_{1}(=U_{3,7})$ & $\calp_3(7)$ \\
\hline
$M_{2}$ & $\calp_3(7) \setminus \{\{1,2,3\}\}$ \\
\hline
$M_{3}$ & $\calp_3(7) \setminus \{\{1,2,3\},\{4,5,6\}\}$ \\
\hline
$M_{4}$ & $\calp_3(7) \setminus \{\{1,2,3\},\{3,4,5\}\}$ \\
\hline
$M_{5}$ & $\calp_3(7) \setminus \{\{1,2,3\},\{1,6,7\},\{3,4,5\}\}$ \\
\hline
$M_{6}$ & $\calp_3(7) \setminus \{\{1,2,3\},\{1,4,5\},\{1,6,7\}\}$ \\
\hline
$M_{7}$ & $\calp_3(7) \setminus \{\{1,2,3\},\{1,4,5\},\{2,4,6\}\}$ \\
\hline
$M_{8}$ & $\calp_3(7) \setminus \{\{1,2,3\},\{1,4,5\},\{2,4,6\},\{3,5,7\}\}$ \\
\hline
$M_{9}$ & $\calp_3(7) \setminus \{\{1,2,3\},\{1,4,5\},\{2,4,6\},\{3,4,7\}\}$ \\
\hline
$M_{10}(=P_7)$ & $\calp_3(7) \setminus \{\{1,2,3\},\{1,4,5\},\{2,4,6\},\{3,4,7\},\{5,6,7\}\}$ \\
\hline
$M_{11}$ & $\calp_3(7) \setminus \{\{1,2,3\},\{1,4,5\},\{2,4,6\},\{3,5,6\}\}$ \\
\hline
$M_{12}$ & $\calp_3(7) \setminus \{\{1,2,3\},\{1,4,5\},\{2,4,6\},\{3,5,6\},\{3,4,7\}\}$ \\
\hline
$M_{13}(=F^{-}_7)$ & $\calp_3(7) \setminus \{\{1,2,3\},\{1,4,5\},\{2,4,6\},\{3,5,6\},\{3,4,7\},\{2,5,7\}\}$ \\
\hline
$M_{14}(=F_7)$ & $\calp_3(7) \setminus \{\{1,2,3\},\{1,4,5\},\{2,4,6\},\{3,5,6\},\{3,4,7\},\{2,5,7\},\{1,6,7\}\}$ \\
\hline
$M_{15}$ & $\calp_3(7) \setminus \{\{1,2,3\},\{1,2,4\},\{1,3,4\},\{2,3,4\}\}$ \\
\hline
$M_{16}$ & $\calp_3(7) \setminus \{\{1,2,3\},\{1,2,4\},\{1,3,4\},\{2,3,4\},\{4,5,6\}\}$ \\
\hline
$M_{17}$ & $\calp_3(7) \setminus \{\{1,2,3\},\{1,2,4\},\{1,3,4\},\{1,6,7\},\{2,3,4\},\{4,5,6\}\}$ \\
\hline
$M_{18}(=O_7)$ & $\calp_3(7) \setminus \{\{1,2,3\},\{1,2,4\},\{1,3,4\},\{2,3,4\},\{1,5,6\},\{2,5,7\},\{3,6,7\}\}$ \\
\end{tabular}
\caption{A list of all $3$-connected matroids of rank $3$ on $\{1,\ldots,7\}$}
\end{center}
\end{table}
First, for $i \in \{1,2,3,4,5,15,16,17\}$, 
let $B_1$, $B_2$ be two bases of $M_i$ and $D=\set{i_1,\ldots,i_6}$ be the multiset consisting of all elements of $B_1 \cup B_2$ with $i_1 \le \cdots \le i_6$.
Then both $\set{i_1,i_3,i_5}$ and $\set{i_2,i_4,i_6}$ are also bases of $M_i$.
Therefore, by \cite{Blum1,Sturmfels1}, it follows that $M_i$ is base-sortable.
Next, for $i \in \{6,7,9,11,18\}$, let $<_i$ be the lexicographic order on $\widetilde{S}_{M_i}$ induced by the following ordering:
\begin{itemize}
\item[$<_6:$] $t_1>\cdots>t_7>x_{236}>x_{126}>x_{245}>x_{136}>x_{457}>x_{257}>x_{156}>x_{125}>x_{127}>x_{345}>x_{146}>x_{267}>x_{456}>x_{124}>x_{357}>x_{347}>x_{356}>x_{247}>x_{147}>x_{157}>x_{467}>x_{235}>x_{237}>x_{246}>x_{367}>x_{567}>x_{346}>x_{256}>x_{137}>x_{134}>x_{234}>x_{135}$ 
\item[$<_7:$] $t_1>\cdots>t_7>x_{357}>x_{356}>x_{247}>x_{267}>x_{467}>x_{567}>x_{256}>x_{134}>x_{156}>x_{137}>x_{234}>x_{456}>x_{367}>x_{135}>x_{167}>x_{157}>x_{257}>x_{237}>x_{347}>x_{146}>x_{124}>x_{147}>x_{126}>x_{457}>x_{245}>x_{345}>x_{346}>x_{125}>x_{127}>x_{236}>x_{235}>x_{136}$, 
\item[$<_{9}:$] $t_1>\cdots>t_7>x_{567}>x_{127}>x_{237}>x_{137}>x_{235}>x_{367}>x_{234}>x_{247}>x_{357}>x_{356}>x_{256}>x_{124}>x_{134}>x_{135}>x_{147}>x_{257}>x_{456}>x_{146}>x_{267}>x_{126}>x_{346}>x_{467}>x_{167}>x_{125}>x_{245}>x_{136}>x_{457}>x_{236}>x_{156}>x_{345}>x_{157}$, 
\item[$<_{11}:$] $t_1>\cdots>t_7>x_{347}>x_{257}>x_{127}>x_{237}>x_{137}>x_{367}>x_{247}>x_{234}>x_{357}>x_{567}>x_{256}>x_{124}>x_{147}>x_{157}>x_{134}>x_{235}>x_{456}>x_{146}>x_{236}>x_{126}>x_{467}>x_{346}>x_{136}>x_{125}>x_{245}>x_{167}>x_{345}>x_{267}>x_{156}>x_{457}>x_{135}$. 
\item[$<_{18}:$] $t_1>\cdots>t_7>x_{146}>x_{137}>x_{167}>x_{126}>x_{136}>x_{246}>x_{236}>x_{247}>x_{245}>x_{135}>x_{456}>x_{357}>x_{237}>x_{457}>x_{145}>x_{267}>x_{256}>x_{345}>x_{125}>x_{157}>x_{346}>x_{347}>x_{127}>x_{235}>x_{356}>x_{567}>x_{467}>x_{147}$. 
\end{itemize}
With the help of \texttt{MACAULAY2}, for a Gr\"obner basis $\calg_{M_i}$ of $J_{M_i}$ with respect to $<_i$, we can check that $\calg_{M_i} \cap S_{M_i}$ is a Gr\"obner basis of $I_{M_i}$ consisting of quadratic binomials.
Moreover, by eliminating some largest variables, we obtain quadratic Gr\"obner bases of $I_{M_j}$ for $j \in \{8,10,12,13,18\}$.
Therefore, by using above results, we have
\begin{theorem}
\thmlab{Main}
\conjref{White_GB} is true for every matroid $M$ with $|E(M)|=7$ except for $M_{14}$ and $M^{\ast}_{14}$.
\end{theorem}
If we can prove that $I_{M_{14}}$ has a quadratic Gr\"obner basis, it follows that \conjref{White_GB} is true for every matroid $M$ with $|E(M)|=7$.

\section*{Acknowledgement}
The authors would like to thank professor Hidefumi Ohsugi for useful comments and suggestions.


\begin{thebibliography}{99}
\bibitem{Blasiak}
J. Blasiak, The toric ieal of a graphic matroid is generated by quadrics, {\it Combinatorica} {\bf 28}(3) (2008) 283-297.

\bibitem{Blum1}
S. Blum, Base-sortable matroids and Koszulness of semigroup rings, {\it Europ. J. Combin.} {\bf 22} (2001), 937-951.

\bibitem{Blum2}
S. Blum, On Koszul algebras, Dissertation, Universit\"at Essen, 2001.

\bibitem{Bonin}
J. Bonin, Basis-exchange properties of sparse paving matroids, {\it Adv. Appl. Math.} {\bf 50} (2013) 6-15.

\bibitem{Macaulay2}
D. R. Grayson and M. E. Stillman, Macaulay2, a software system for research in algebraic geometry, Available at \texttt{http://www.math.uiuc.edu/Macaulay2/}.

\bibitem{HH}
J. Herzog and T. Hibi, Discrete polymatroids, {\it J. Algebraic Combin.} {\bf 16} (2002) 239-268.

\bibitem{Hibi}
T. Hibi (ed.), ``Gr\"obner Bases: Statistics and Systems," Springer, 2013.

\bibitem{Kashiwabara}
K. Kashiwabara, The toric ideal of a matroid of rank $3$ is generated by quadrics, {\it Electron. J. Combin.} {\bf 17} (2010) \#R28.

\bibitem{LaMi}
M. Laso\'n and M. Micha\l ek, On the toric ideal of a matroid, {\it Adv. Math.} {\bf 259} (2014) 1-12.


\bibitem{MayRoy}
D. Mayhew and G. F. Royle, Matroids with nine elements, {\it J. Combin. Theory Ser. B} {\bf 98} (2008) 415-431.


\bibitem{OH}
H. Ohsugi and T. Hibi, Compress polytopes, initial ideals and complete multipartite graphs, {\it Illinois J. Math.} {\bf 44}(2) (2000) 141-146.

\bibitem{Oxley}
J. Oxley, ``Matroid theory," Second Edition, Oxford University Press, New York, (2011).

\bibitem{Schweig}
J. Schweig, Toric ideals of lattice path matroids and polymatroids, {\it J. Pure Appl. Algebra} {\bf 215}(11) (2011) 2660-2665.

\bibitem{Shibata}
K. Shibata, Toric ideals of series and parallel connections of matroids,
{\it J. Algebra Appl.} {\bf 15} (2016), no. 6, 1650106, 11 pages.

\bibitem{Sturmfels1}
B. Sturmfels, ``Gr\"obner bases and convex polytopes," Amer. Math. Soc., Providence, RI, 1996.


\bibitem{White1}
N. White, The basis monomial ring of a matroid, {\it Adv. Math.} {\bf 24} (1977), 292-297.

\bibitem{White2}
N. White, A unique exchange property for bases, {\it Linear Algebra Appl.} {\bf 31} (1980), 81-91.
\end{thebibliography}
     \end{document}